\documentclass{elsarticle}
\usepackage{tikz}
\usepackage{lipsum}
\usepackage{amsthm}
\usepackage{amsmath}
\usepackage{amssymb}
\usepackage{mathtools}
\usepackage{enumerate}
\usepackage{textcomp}

\makeatletter
\def\ps@pprintTitle{%
 \let\@oddhead\@empty
 \let\@evenhead\@empty
 \def\@oddfoot{}%
 \let\@evenfoot\@oddfoot}
\makeatother

\usepackage{graphicx}
\usepackage{amssymb}

\newtheorem{theorem}{Theorem}[section]
\newtheorem{corollary}{Corollary}[section]
\newtheorem{lemma}{Lemma}[section]
\newtheorem{definition}{Definition}[section]
\newtheorem{claim}{Claim}[section]

\usepackage{lineno}

\begin{document}

\begin{frontmatter}


\title{Modal logic with
the difference modality of topological $T_0$-spaces\footnote{This publication was prepared within the framework Academic Fund Program at the National Research University Higher School of Economics (HSE) in 2019 (grant \textnumero 19-04-050 ) and the Russian Academic Excellence Project "5-100"}}

\author[rvt]{Rajab Aghamov}

\address{Higher School of Economoics \\6 Usacheva st., Moscow,  119048, 

agamov@phystech.edu}

\begin{abstract}
The aim of the paper is to study the topological modal logic of $T_0$ spaces, with the difference modality (for $T_n$, where $n\geq1 $ the corresponding logics were known). We consider propositional modal logic with two modal operators $\square$ and $[\ne]$. $\square$ is interpreted as an interior operator and $[\ne]$ corresponds to the inequality relation. We introduce the logic $S4DT_0$ and show that  $S4DT_0$ is the logic of all $T_0$ spaces and has the finite model property.
\end{abstract}


\end{frontmatter}

\section*{Introduction} 

In this paper we study the topological semantics of modal logics. Several interpretations of the modal box as an operator over a topological space are possible. Namely diamond-as-closure-operator and diamond-as-derivation-operator have been pioneering in the semantics of modal logic as far back as in 1944, in the celebrated paper of McKinsey and Tarski (cf. \cite{7}). They showed that \textbf{S4} is the logic of all topological spaces and the logic of any metric dense-in-itself space is \textbf{S4}. 
This remarkable result also demonstrates a relative weakness of the interior operator to distinguish between interesting topological properties

The second interpretation gives more expressive power. $T_0$ and $T_D$ separation axioms become expressible (cf. \cite{1}, \cite{17}); the real line can be distinguished from the rational line (cf. \cite{21}). It also has its limitations (for example, it is still impossible to distinguish $\mathbb{R}^2 $ from $\mathbb{R}^3 $). 

We can increase the expressive power by adding extra modalities 
 (cf. \cite{22}, \cite{23}). For example, connectedness is expressible in modal logic with the interior and the universal modality (cf. \cite{24}) and $T_1$ separation axiom becomes expressible in modal logic with the interior and the difference modality (cf. \cite{25}). 

In this paper we add the difference modality (or modality of inequality) $[\ne]$, interpreted as ''true everywhere except here'' .  The expressive power of this language in topological spaces has been studied by Gabelaia in \cite{15}, where he presented an axiom that defines $T_0$ spaces. 

The first section contains basic information, definitions and results from the theory of modal logics and general topology.

In the last two sections we formulate completeness of $\mathbf{S4DT_0}$ logic with respect to $T_0$ spaces respectively and the finite modal property.

\section{Language, axioms and logic}

In this paper, we study propositional modal logics with two modal operators, $\square$ and $[\ne ]$. A formula is defined as follows:

$$\phi ::= p \,|\,\bot \,|\,\phi\rightarrow\phi\,|\,\square\phi\,|\,[\ne]\phi$$

The classic logic operators $\vee, \wedge, \neg, \top, \equiv$ are expressed in terms of $\rightarrow$ and $\bot$ in the standard way. The dual modal opertors $\Diamond$ and $\langle \neq \rangle $ are defined as usual: $\Diamond \phi = \neg\square\neg \phi $, $\langle \neq \rangle \phi = \neg[\ne]\neg \phi$ respectively. We denote $[\ne]\phi\wedge \phi$ by $[\forall]\phi$.

The set of all bimodal formulas is called the $bimodal\; language$ and is denoted by $\mathcal{ML}_2 $.

\begin{definition}

A ($normal\; bimodal)\, logic$ is a set of formulas {\itshape L} $\subseteq \mathcal{ML}_2$ such that:

{\bf 1.} $L$ contains all the classical tautologies.

{\bf 2.} $L$ contains the modal axioms of normality:

\begin{center}

$\square(p\rightarrow q)\rightarrow (\square p\rightarrow \square q ) $, 

$[\ne](p\rightarrow q)\rightarrow ([\ne] p\rightarrow [\ne] q ) $.
\end{center}

{\bf 3.} $L$ is closed with respect to the following inference rules:
\begin{center}

$\frac{\phi\rightarrow \psi,\; \phi}{\psi}$ (MP),

$\frac{\phi}{\square \phi}\, (\rightarrow\square)$,

$\frac{\phi}{[\ne] \phi} (\rightarrow [\ne])$,

$\frac{\phi}{\lbrack \psi/p \rbrack \phi}$ (Sub).

\end{center}

\end{definition}

Let $L$ be a logic and $\Gamma$ be a set of formulas. The  minimal logic containing $L\cup \Gamma$ is denoted by $L+\Gamma$. We also write $L + \psi $ instead of $L + \{\psi \}$.

In this paper we will use the following axioms:

{\bf ($T_{\square}$)} \quad $\square p\rightarrow p$,

{\bf ($4_{\square}$)} \quad $\square p\rightarrow \square\square p$,

{\bf ($D_{\square}$)} \quad $\lbrack\forall\rbrack p\rightarrow\square p$,

{\bf ($B_{D}$)} \quad $p\rightarrow\lbrack\neq\rbrack\langle \neq \rangle p$,

{\bf ($4_{D}$)} \quad $\lbrack\forall\rbrack p \rightarrow\lbrack\neq\rbrack\lbrack\neq\rbrack p$,

{\bf ($AT_{0}$)} $\quad (p\wedge \lbrack \neq \rbrack \neg p \wedge \langle \neq \rangle (q\wedge \lbrack \neq \rbrack \neg q))\rightarrow(\square\neg q\vee\langle \neq \rangle(q\wedge\square\neg p))$. 

We introduce the notation for the following logics:
\begin{center}
\textbf{S4} = $\mathbf{K_1}$ +$T_{\square}$ + $4_{\square}$

$\mathbf{S4D}$ = $\mathbf{K_2}$ + $T_{\square}$ + $4_{\square}$ + $D_{\square}$ + $B_{D}$ + $4_{D}$ 

$\mathbf{S4DT_0}$ = \textbf{S4D} + $AT_0$ 
\end{center}

\section{Topological semantics}
\begin{definition}
A \emph{topological space} is a pair $\mathbb{X} = (X, \Omega)$ where $X$ is a nonempty set (the $domain$ of the space) and $\Omega$  is a set of subsets of $X$ satisfying the following properties:

$\bf{1.}$ The empty set $\emptyset$ and $X$ itself belong to $\Omega$.

$\bf{2.}$ The union of any collection of sets from $\Omega$ is contained in $\Omega$.

$\bf{3.}$ The intersection of any finite collection of sets from $\Omega$ is also contained in $\Omega$.

\end{definition}

The elements of $\Omega$ are called $open\; sets$ and $\Omega$ is called a $topology$ on $X$. If $(X, \Omega)$ is a topological space and $x$ is a point in $X$, a $neighbourhood$ of $x$ is an open set $U$ containing $x$. A $closed\,\, set$ is a set whose complement is an open set.

\begin{definition}
The \emph{interior} of a set $A$ in a topological space $\mathbb{X}$ is the greatest (with respect to inclusion) open set in $\mathbb{X}$ contained in $A$, i.e., an open set that contains any other open subset of $A$.  It is denoted by Int $A$.
\end{definition}

\begin{definition}

The closure of a set $A$ is the smallest closed set containing $A$.  It is denoted Cl $A$.

\end{definition}

\begin{definition}

Let $\mathbb{X}$ be a topological space, then $\mathbb{X}$ is an \emph{Alexandroff space} if arbitrary intersections of open sets are open.

\end{definition}

\begin{definition}

 A topological space $\mathbb{X}$ is a $T_0$-space if for every pair of distinct points of $X$, at least one of them has a neighborhood not containing the other.

\end{definition}

\begin{definition} A \emph{topological model} on a topological space $\mathbb {X}: = (X, \Omega) $ is a pair $ (\mathbb{X}, V)$, where $V: PV \rightarrow P(X)$ (the set of all subsets), i.e.$\:$ a function that assigns each propositional variable to $p$ a set $V(p) \subseteq X $ and is called a $valuation$. The truth of a formula $ \phi $ at a point $ x $ of a topological model $ \mathcal {M} = (\mathbb {X}, V) $ (notation: $ \mathcal{M}, x \vDash \phi $) is defined by induction:

$\mathcal{M}, x\vDash p \Leftrightarrow x\in V(p) $,

$\mathcal{M}, x\nvDash \bot  $,

$\mathcal{M}, x\vDash \phi\rightarrow\psi \Leftrightarrow \mathcal{M}, x\nvDash \phi $ or $\mathcal{M}, x\vDash \psi $,

$\mathcal{M}, x\vDash \square \phi\Leftrightarrow\exists U\in \Omega (x \in U$ and $\forall y \in U (\mathcal{M}, y\vDash\phi)$, 

$\mathcal{M}, x\vDash \lbrack \neq \rbrack \phi\Leftrightarrow \forall y\neq x (\mathcal{M}, y\vDash\phi)$. 

\end{definition}

\begin{definition} Let $ \mathcal {M} = (X, \Omega, V)$ be a topological model and $ \phi $ be a formula. We say that the formula $\phi$ is true in the model $ \mathcal {M} $ (notation: $ \mathcal {M} \vDash \phi $), if it is true at all points of the space, i.e.

\begin{center}

$\mathcal{M}\vDash\phi\Leftrightarrow\forall x \in X, \mathcal{M}, x\vDash\phi$.

\end{center}

\end{definition}

\begin{definition}
Let $ \mathbb {X} = (X, \Omega) $ be a topological space, $ \mathcal {C} $ be a class of spaces and  $ \phi $ be a formula. We say that a formula is valid in $ \mathbb {X} $ (notation: $ \mathbb {X}\vDash \phi $) if it is true in every model on this topological space, i.e.

\begin{center}

$\mathbb{X}\vDash\phi\Leftrightarrow\forall V \; (\mathbb{X}, V\vDash\phi)$.

\end{center}

\end{definition}

{\noindent}We say that the formula $ \phi $ is valid in $\mathcal {C} $ if it is valid in every space in $ \mathcal {C} $.

\begin{definition} The logic of a class of topological spaces $ \mathcal {C} $ (denoted by $ Log (\mathcal {C})) $ is the set of all formulas of the language $ \mathcal {ML}_2 $ that are valid in all spaces of the class $ \mathcal{C} $.

\end{definition}

\begin{theorem} (c.f. \cite {20}). Let $\mathcal{C}$ be a class of topological spaces. Then $ Log(\mathcal{C})$ is a modal logic.
\end{theorem}

\begin{lemma}

Let $\mathcal{M} = (X,\Omega, V)$ be a topological model. Denote $V(\phi) = \{x\in X\;|\;\mathcal{M}, x\vDash\phi\}$, where $\phi$ is an arbitrary formula. Then we have:

{\bf 1.} $V(\phi\vee\psi) = V(\phi)\cup V(\psi)$

{\bf 2.} $V(\phi\wedge\psi) = V(\phi)\cap V(\psi)$

{\bf 3.} $V(\neg\phi) = X - V(\phi)$

{\bf 4.} $V(\phi\rightarrow\psi) = (X - V(\phi))\cup V(\psi)$

{\bf 5.} $\mathcal{M}\vDash\phi\rightarrow\psi\Leftrightarrow V(\phi)\subseteq V(\psi)$

{\bf 6.} $V(\top) = X$

{\bf 7.} $V(\square\phi) = Int(V(\phi)) $

{\bf 8.} $V(\Diamond\phi) = Cl(V(\phi)) $

\end{lemma}
\begin{proof}
Since the first 6 points are obvious, so we prove only last 2 points.

{\bf 8}. \,$\,x \in V(\square\phi)$

$\Leftrightarrow\mathcal{M}, x\vDash\square\phi$

$\Leftrightarrow\exists U\in\Omega(x\in U\;\&\;\forall y\in U(\mathcal{M}, y\vDash\phi))$

$\Leftrightarrow\exists U\in\Omega(x\in U\;\&\;\forall y\in U(y\in V(\phi))$

$\Leftrightarrow\exists U\in\Omega(x\in U\;\&\; U\subseteq V(\phi))$

$\Leftrightarrow x\in Int(V(\phi))$

9. \,$\,x\in V(\Diamond\phi)$

$\Leftrightarrow\mathcal{M}, x\vDash\Diamond\phi$

$\Leftrightarrow\mathcal{M}, x\vDash\neg\square\neg\phi$

$\Leftrightarrow x \in X - Int(X - v(\phi))$

$\Leftrightarrow x\in Cl(V(\phi))$

\end{proof}

\begin{lemma}
\textit{Let $\mathbb{X} = (X, \Omega)$ be a topological space then  $\mathbb{X} \vDash AT_0$ iff $\mathbb{X}$
is a $T_0$ space.}
\end{lemma}

\begin{proof}

($\Rightarrow$) We prove by contradiction. Assume $\mathbb{X} \vDash AT_0$ and let there be points $x\neq y $ such that $\forall U\in\Omega,\, x\in U \Leftrightarrow y\in U$. Define a valuation $V$ such that $V(p) = \{x\}$ and $V(q) = \{y\}$. Then $\mathbb{X}, V, x \vDash p\wedge \lbrack \neq \rbrack \neg p \wedge \langle \neq \rangle (q\wedge \lbrack \neq \rbrack \neg q)$ and $\mathbb{X}, V, x \nvDash \square\neg q\vee\langle \neq \rangle(q\wedge\square\neg p)$. This contradicts the fact that $\mathbb{X}\vDash AT_0$. 

($\Leftarrow$)Assume $\mathbb{X}$ is a $T_0$ space. Let $\mathbb{X}, V, x\vDash p\wedge \lbrack \neq \rbrack \neg p \wedge \langle \neq \rangle (q\wedge \lbrack \neq \rbrack \neg q)$. Then there is a point $y$, such that $V(q) = \{y\}$. Further, at least one of the points $x$ and $y$ is contained in a neighborhood that does not contain the other. That means $\mathbb{X}, V, x \vDash \square\neg q$ or $\mathbb{X}, V, y \vDash \square\neg p$ which proves our assertion.

\end{proof} 

\begin{definition} A logic $L$ is $complete$ with respect to a class of topological spaces $\mathcal C$ if $Log(\mathcal C) = L$.
\end{definition}

\begin{theorem}(cf. \cite{20}) The logic $S4D$ is complete with respect to all topological spaces.
\end{theorem}

\section{Kripke semantics.}

\begin{definition} A $Kripke\, frame$ is a tuple  ${\displaystyle \langle W,R_1,\ldots, R_n\rangle }$, where $ W\ne\emptyset \,$ is a set, and $R_i$ (for $i = 1,\ldots,n$) is a binary relation on $W$. Elements of $W$ are called \emph{points or worlds}, and $R_i$ for $i = 1,\ldots,n$ is an \emph{accessibility relation}.
\end{definition}
In this article we will deal with Kripke frames with two binary relations. The first relation will be denoted by $R$, the second by $R_D$.

\begin{definition} A $valuation$ on a Kripke frame $ F = (W,\: R_1,\: R_2,\: ...,\: R_n) $ is a function $ V: PV \longrightarrow 2^W $. A \emph{Kripke model} is a pair $ M = (F, V) $. Then we inductively define the notion of a formula $\phi$ being true in $M$ at a point $x$  as follows:

\begin{center}

$M, x\vDash p \Leftrightarrow x \in V(p)$, for $p\in PV$

$M, x\nvDash \bot$

$M, x\vDash \phi\rightarrow \psi \Leftrightarrow M,\; x\nvDash \phi$ or $M,\; x\vDash \psi$

$M, x\vDash \square_{i} \phi \Leftrightarrow \forall y(xR_{i}y\Rightarrow M, y\vDash \phi) $

\end{center}

\end{definition}

 For a subset $U \subseteq W\: M, U\vDash \phi$ denotes that for any $x\in U (M, x\vDash \phi).$ We say that a formula $ \phi $ is $valid$ in a model $M$ (notation: $ M \vDash \phi $), if $ \forall x \in W (M, x \vDash \phi) $. We also say that a formula $ \phi $ is $valid$ on a frame $F $(notation: $ F \vDash \phi $) if it is valid in all models of the frame $F$ and a formula is $valid$ in a class of frames if it is valid in every frame from this class.

\begin{definition} The $logic$ of a class of frames $\mathcal{C}$ (in notation $Log(\mathcal{C})$) is the set of formulas that are valid in all frames from $\mathcal{C}$. For a single frame $F$, $Log(F)$ stands for $Log(\{F\})$.
\end{definition}
\begin{definition} A logic $L$ is called Kripke complete if there exists a class of frames $ \mathcal{C}$, such that $L = Log(\mathcal{C}) $.
\end{definition}
\begin{definition} Let $L$ be a modal logic. A frame $F$ is called an $L$-frame if $ L \subseteq Log(F) $.
\end{definition}
\begin{theorem} (c.f. \cite {3}). Let $F$ be a Kripke frame. Then $ Log(F) = \{\phi \; | \; F  \vDash \phi \} $ is a modal logic.
\end{theorem}
Let us introduce some notations. Let $ W $ be an arbitrary nonempty set, $ B \subseteq W $; $ R, \; R '\subseteq W \times W $ are relations on $ W $.

\begin{center}

$R|_{B}\rightleftharpoons R\cap(B\times B)$;

$Id_{W}\rightleftharpoons \{(x,x)|x\in W\}$;

$R^{+}\rightleftharpoons R\cup Id_{W} $ (reflexive closure);

$R\circ R'\rightleftharpoons \{(x,z)|\exists y(xRy \;\&\; yR'z )\}$;

$R^{0}\rightleftharpoons Id_{W}$;

$R^{n+1}\rightleftharpoons R^{n}\circ R$;

$R^{*}\rightleftharpoons\bigcup_{n=0}^{\infty} R^{n}$(reflexive and transitive  closure).

\end{center}

Let $F = (W, R_1, ..., R_n)$ be a frame, and let $ x \in W $.  $ R_ {i} (x) = \{y \; | \; xR_ {i} y \}$,$\:$ $R_{i}^{-1}(x) = \{y \; | \; yR_{i}x \}$.$\:$Let $ U \subseteq W $, then $ R_{i}(U) = \bigcup_{x \in U} R_ {i}(x) $, $R_{i}^{-1} (U)= \bigcup_ {x\in U} R_{i}^{-1}(x)$.

\begin{definition} Let $ F = (W, R_1, ..., R_n) $ be a Kripke frame and $ S ^ {*} $ be the transitive and reflexive closure of the relation $S = (\bigcup_{i =0}^{n}R_{i}) $. For $x\in W, \; W^{x} \rightleftharpoons \{y \; | \; xS^{*} y \} $ (the set of all points reachable from the point $x$ by relation $S^{*}$). Frame $F^{x} = (W^{x}, R_{1}|_{W ^{x}}, ..., R_{n}|_{W ^ {x}}) $ is called a\emph{ generated subframe (cone)}.
\end{definition}

\begin{lemma} Let $F = (W, R_1, R_2, ..., R_n)$ be a Kripke frame and $ \mathcal{C} $ be a class of Kripke frames, then 

$1. \,Log(F) = \bigcap_{x\in W} Log(F^{x}) = Log (\{ F^{x} \,|\, x \in W \})$.

$2. \,Log(\mathcal{C}) = Log(\{F^{x} \,|\, F\in \mathcal{C}, x \in F \})$

\end{lemma}

Let $ F = (W, R)$ be an $\mathbf{S4}$-frame, then the set of subsets $ T = \{U \subseteq W \,|\, R(U) \subseteq U \} $ defines a topology on W. Topological space $ (W, T) $ is denoted by $Top(F)$.$\,$ This topology is $Alexandroff$ (see \cite{20}) (since $R(x)$ is the minimal open neighborhood of $x$ for any $x$). 

Consider the interpretation of the language $\mathcal{ML}_2$ in topological spaces with a binary relation of the form $(\mathbb{X}, R)$, where $\square$ is interpreted in the same way as in topological semantics, and $[\ne]$ using $R$ as in Kripke semantics.$\, $

If the reflexive closure of the binary relation $R$ is the  universal relation (i.e., $R\cup Id_{W} = W\times W$), then the relation $R$ can be characterized by the set of all irreflexive points, which we call $selected\; points$.

The following lemma is well-known (cf. \cite{2}, \cite{3})

\begin{lemma}
(see \cite{1}, \cite{2}) Let $F = (W, R, R_D)$ be a Kripke frame, then

{\bf 1.} $F\vDash B_{D}\Leftrightarrow \forall x, y\in W \;(xR_{D}y \Rightarrow yR_{D}x)\Leftrightarrow R_D$ is symmetryc;

{\bf 2.} $F\vDash 4_{D}\Leftrightarrow R_{D}^{2} \subseteq R_{D}\cup Id_W \Leftrightarrow R_D \cup I_d$ is transitive;

{\bf 3.} $F\vDash 4_{\square}\Leftrightarrow \forall x, y, z\in W \;(xRy \;\&\; yRz\Rightarrow xRz)\Leftrightarrow R$ is transitive;

{\bf 4.} $F\vDash T_{\square}\Leftrightarrow \forall x\in W \;xRx \Leftrightarrow R$ is reflexive;

{\bf 5.} $F\vDash D_{\square}\Leftrightarrow R \subseteq R_{D}\cup Id_W$.
\end{lemma}

Now let $F = (W, R, R_D)$ be an $\mathbf{S4D}$-cone, so $R_D \cup I_{d_W} = W\times W $. We define a space with selected points $Top_{D}(F) \rightleftharpoons (Top(F), A) $, where $ A = \{v \; | \; \neg vR_{D} v \} $. Note that we may consider topological space $\mathbb{X}$ as ($\mathbb{X}, X$), where the domain $X$ is the set of selected points, in other words all the points are selected.

\begin{lemma} Let $(F, V)$ be a model on $F$, where $F = (W, R, R_D)$ is an $\mathbf{S4D}$-cone, then
$$F,V,x\vDash\phi\Leftrightarrow Top_{D}(F), V, x\vDash\phi, $$
for any $ x \in W $ and for any formula $ \phi $.
\end{lemma}
\begin{proof}The standard proof is carried out by the induction on the length of the formula. 
\end{proof}

\begin{corollary} Let $F=(W, R, R_D)$ be an $\mathbf{S4D}$-cone, then
$$Log(F) = Log(Top_{D}(F)).$$
\end{corollary}

\begin{lemma}
Let $F = (W, R, R_D)$ be an $\mathbf{S4D}$-cone, then: 
\begin{center}
$F \vDash AT_0  \Longleftrightarrow  \forall x, y \in W (x \neq y \wedge xRy \wedge yRx \Longrightarrow xR_Dx \vee yR_Dy)$
\end{center}
\end{lemma}
\begin{proof}
 Suppose there are two points $x$, $y$ such that they both are irreflexive with respect to the second relation ($R_{D}$-irreflexive) and mutually reachable by the first relation. We define a model $M = (F, V) $ by defining valuation as follows: $ V(p) = \{ x \}, V(q) = \{ y \}$. Then
\begin{center}

$M, x \models p\wedge \lbrack \neq \rbrack \neg p \wedge \langle \neq \rangle (q\wedge \lbrack \neq \rbrack \neg q)$
and
\end{center}
\begin{center}

$M, x \nvDash\square\neg q\vee\langle \neq \rangle(q\wedge\square\neg p)$. 

\end{center}

Conversely, suppose $ M, x\models p\wedge\lbrack\neq\rbrack\neg p\wedge\langle\neq\rangle(q\wedge\lbrack \neq\rbrack\neg q) $. Then $ V(p) = \{x\}$ and there is a point $y$ such that $V(q) = \{y\}$. These points are $R_{D}$-irreflexive. By assumption, they are not mutually accessible by the first relation. We consider 2 cases: 

1. $\neg xRy$. Then

\begin{center}

$M, y \models \square\neg q  \Rightarrow  M, x \models\square\neg q\vee\langle \neq \rangle(q\wedge\square\neg p)$. 
\end{center}

2. $\neg yRx$. Then

\begin{center}

$M, y \models q\wedge\square\neg p \Rightarrow M, x \models \langle \neq \rangle(q\wedge\square\neg p) \Rightarrow  M, x \models\square\neg q\vee\langle \neq \rangle(q\wedge\square\neg p)$. 
\end{center}
\end{proof}

\section{p-morphism}

For two topological spaces $\mathbb{X}$ and $\mathbb{Y}$ a map $f: \mathbb{X} \rightarrow \mathbb{Y}$ is said to be $continuous$ if for every open subset $U\subset \mathbb{Y}$, the inverse image $f^{-1}(U)\subset \mathbb{X}$ is open in $\mathbb{X}$. Map $f$ is said to be $open$ if for every open set $U$ in $\mathbb{X}$, $f(U)$ is open in $\mathbb{Y}$. We call $f$ $interior$ if it is both open and continuous. 

\begin{definition}A map between topological spaces $ f: \mathbb{X}\rightarrow \mathbb{Y} $ is called $p$-$morphism$ if it is surjective and interior (notation: $f:\mathbb{X} \twoheadrightarrow \mathbb{Y}$).
\end{definition}

\begin{definition}A map between topological spaces with selected points $ \mathcal{X} = (\mathbb{X},\, A_{\mathbb{X}}) $ and $ \mathcal{Y} = (\mathbb{Y},\, A_{\mathbb{Y}}) $ is called a p-morphism if it is a p-morphism of topological spaces $f:\mathbb{X}\twoheadrightarrow \mathbb{Y}$, and $$ A_{\mathbb{Y}} = \{y \; | \; \exists x \in A_{\mathbb{X}}\, (f^{-1}(y) = \{x \}) \}$$
\end{definition}
\begin{lemma}(cf.\cite{12})For a given map $f:X\rightarrow Y$ the following statements are equivalent:

1. $f$ is interior;

2. $ f^{-1}(Int\,Z) = Int\,f^{-1}(Z)$ for any $Z\subseteq Y$;

3. $f^{-1}(Cl\,Z) = Cl\,f^{-1}(Z)$ for any $Z\subseteq Y$.
\end{lemma}

\begin{lemma} Let $ \mathcal{X} = (\mathbb{X},\, A_{\mathbb{X}}) $ and $\mathcal{Y} =  (\mathbb{Y},\, A_{\mathbb{Y}}) $ be topological spaces with selected points and $ f: \mathbb {X} \twoheadrightarrow \mathbb {Y} $ be a p-morphism. Let $ V_{\mathbb {Y}}$ be a valuation on topological space  $\mathbb{Y}$ and $ V_{\mathbb {X}} (p) = f ^ {-1} (V_{\mathbb {Y}}(p))$ for all $p \in PV $. Then for any formula $\phi $ the following holds:
$$\forall x \in \mathbb{X}\; (\;\mathcal{X}, V_{\mathbb{X}}, x\vDash\phi\Leftrightarrow \mathcal{Y}, V_{\mathbb{Y}}, f(x)\vDash\phi).$$
\end{lemma}
\begin{proof} The statement of the lemma can be rewritten as follows:

\begin{center}
    
    $V_{\mathbb{X}}(\phi) = f^{-1}(V_{\mathbb{Y}}(\phi)) $
    
\end{center}

The proof proceeds in a straightforward way by induction on the length of the formula. Let us consider cases when $\phi = \square\psi$ and $\phi = [\ne]\psi$. The other cases are trivial.

Suppose that $\phi = \square\psi$. First, we use the assertion of the lemma 2.1, then the previous lemma.

$f^{-1}(V_{\mathbb{Y}}(\square \psi))  = f^{-1}(Int\,  (V_{\mathbb{Y}}(\psi))) = Int\, (f^{-1}(V_{\mathbb{Y}}(\psi))) \stackrel{\text{IH}}= Int\,(V_{\mathbb{X}}(\psi)) = V_{\mathbb{X}}( \square \psi)$.

Now we turn to $[\ne]$. Suppose $\mathcal{Y}, V_{\mathbb{Y}}, f(x)\vDash[\ne]\psi$. If $f(x)\in A_{\mathbb{Y}}$, then $$\forall z \ne f(x)\,(\mathcal{Y}, V_{\mathbb{Y}}, z\vDash\psi)\xRightarrow[]{IH} \forall y \ne x\,(\mathcal{X}, V_{\mathbb{X}}, y\vDash\psi)\Rightarrow\mathcal{X}, V_{\mathbb{X}}, x\vDash\ [\ne]\psi.$$ If $f(x)\notin A_{\mathbb{Y}}$, then $$\forall z \,(\mathbb{Y}, V_{\mathbb{Y}}, z\vDash\psi)\xRightarrow[]{IH}\forall y \,(\mathcal{X}, V_{\mathbb{X}}, y\vDash\psi)\Rightarrow\mathcal{X}, V_{\mathbb{X}}, x\vDash\ [\ne]\psi.$$
Suppose $\mathcal{X}, V_{\mathbb{X}}, x\vDash[\ne]\psi$. There are 3 cases: 

1. $x\in A_{\mathbb{X}} \: \& \: f(x)\in A_{\mathbb{Y}}$

2. $x\notin A_{\mathbb{X}} \:\&\: f(x)\notin A_{\mathbb{Y}}$

3. $x\in A_{\mathbb{X}} \:\&\: f(x)\notin A_{\mathbb{Y}}$

The first two cases are obvious, so we only consider the last case. By the definition of p-morphism, there is another point $x'$, such that $f(x) = f(x')$. Then $$\mathcal{X}, V_{\mathbb{X}}, x\vDash [\ne]\psi\Rightarrow \mathcal{X}, V_{\mathbb{X}}, x^{\prime}\vDash\psi \xRightarrow[]{IH} \mathcal{Y}, V_{\mathbb{Y}}, f(x)\vDash\psi\Rightarrow \mathcal{X}, V_{\mathbb{X}}, x\vDash \psi.$$ 

It follows that, 
$$\forall z(\mathcal{X}, V_{\mathbb{X}}, z\vDash\psi)\xRightarrow[]{f-surjective}\forall y( \mathcal{Y}, V_{\mathbb{Y}}, y \vDash\psi)\Rightarrow\mathcal{Y}, V_{\mathbb{Y}}, f(x) \vDash[\ne]\psi.$$ \end{proof}

\section{Canonical frames and Kripke completeness}

The axioms $T_{\square},  4_{\square}, D_{\square}, B_{D}, 4_{D}$ are Sahlqvist formulas, so by the Sahlqvist theorem we obtain the canonicity and Kripke completeness for logic $\mathbf{S4D}$ (see \cite{2}).
To prove the Kripke completeness of logic $\mathbf{S4DT_0}$, we use the canonical model construction.

\begin{definition} The $canonical\; frame$ for $L$ is $F_L = (W_L, R_{1,L}, \dots, R_{n,L})$,  where $W_L$ is  the  set  of  all  maximal consistent theories over $L$ (see \cite{19}) and $xR_{iL}y$ if for every $\square_{i} A\in x$ we have $A\in y$.
\end{definition}
\begin{definition} The canonical model for $L$ is a model $M_L$ on the frame $F_L$ with a valuation function $V_L$ such that $V_{L}(p) = \{x \,|\, p \in x\}.$
\end{definition}
\begin{theorem}($\bf{Canonical\: Model\: Theorem}$, cf. \cite{2}, \cite{3}) For the modal logic $L$ and its canonical model $ M_{L} = (W_L, R_{1,L}, \dots, R_{n,L}, V_L) $, it is true that, $\forall\phi\;\forall x \in W$

i.\,$M_L, x\vDash \phi\Leftrightarrow \phi \in x$

ii.\,$M_L\vDash \phi\Leftrightarrow \phi \in L$
\end{theorem}
\begin{lemma}$\mathbf{S4DT_0}$ logic is Kripke complete.
\end{lemma}
\begin{proof}We take the canonical model $ M_L = (F_L, V_L) $ of logic $\mathbf L = {S4DT_0}$. By the Sahlqvist theorem $ F $  is an $\mathbf{S4D}$-frame (see \cite{2} and \cite{14}). Consider a cone $ M = M_{L}^x $ and assume there exist different points $z$ and $y$ such that they are $R_{D}$-irreflexive  $ (\neg zR_{D}z $ and $ \neg yR_{D}y) $ and are mutually reachable by the first relation (i.e. $ zRy \wedge yRz $). Note that, by definition of the canonical model $( \neg zR_{D}z \Longleftrightarrow \exists \phi (\lbrack \neq \rbrack \phi \in z \; \& \; \phi \notin z) $. But on the other hand $ \phi $ is true at all the other points of the cone. Hence, $ \neg \phi $ is false everywhere, except for the point $ z $. Similarly, it can be shown that there exists a formula $ \psi $ that is true only at $ y $. Hence,

$$
M, z \models\neg\phi\wedge[\neq]\phi \wedge\langle \neq \rangle(\psi\wedge[\neq]\neg \psi)
$$
On the other hand,

$$M, z \nvDash\square\neg \psi \vee\langle \neq \rangle(\psi\wedge\square\phi).$$
As a result, assuming the opposite, we get $M, z\nvDash AT_0$, which contradicts to the previous theorem.
\end{proof}

\section{Finite model property}

A formula is $satisfiable\,\, in\,\, a \,\, frame\,\, F\,(in\,\, a \,\,class\,\,of\,\, frames\,\, \mathbb{C})$ if it is true at a point of some model over $F $ (over some frame in $\mathbb{C})$.

\begin{definition}

Logic $L$ has the $finite\, model\, property$ if $L = L(\mathbb{C})$ for some class of finite frames $\mathbb{C}$.

\end{definition}
Logic $L = L(\mathbb{C})$ has the finite model property if and only if each satisfiable in $\mathbb{C}$ formula is satisfiable in a finite $L$-frame.

\begin{definition}

Let us consider a frame $F = (W, R_1, R_2)$ and an equivalence relation $\sim$ on $W$. A frame $F/{\sim} = (W/{\sim},\, R_1/{\sim},\,R_2/{\sim} )$ is said to be the $minimal \,\, filtration$ of $F$ through ${\sim}$, if for $U_1, U_2\in W/{\sim}$ and $i = 1, 2$
\begin{center}
    $U_1R_i/{\sim} U_2 \Leftrightarrow \exists u\in U_1\, \exists v \in U_2\, uR_iv$
\end{center}

\end{definition}

\begin{definition}

Let M be a model and $\phi$ be a formula. We define the equivalence $\sim_{\phi}\,\, induced\,\,by\,\,the\,\,formula\,\,\phi$ on the points of M as follows: $u\sim_{\phi} v$ iff every subformula of $\phi$ is simultaneously true or false in u and in v.

\end{definition}

We say that an equivalence $\sim$ $\textit{agrees with  formula $\phi$ in a model}$ if $\sim \subseteq\sim_{\phi}$.

\begin{lemma}(cf. \cite{13})
If a formula $\phi$ is satisfiable in a model $M$ over a frame $F$ and an  equivalence ${\sim}$ agrees with formulas $\phi$, then $\phi$ is satisfiable in $F/{\sim}$.
\end{lemma}

 A \emph{partition of a set} $W$ is a family of disjoint subsets of $W$ whose union coincides with $W$. If $\mathbb{A}$ and $\mathbb{B}$ are partitions of a set $W$ and each element of $\mathbb{A}$ is a subset of one element from $\mathbb{B}$, then we say $\mathbb{A}$ is a \emph{refinement} of $\mathbb{B}$. We denote by ${\sim}_\mathbb{A}$ the equivalence relation whose set of classes coincides with $\mathbb{A}: \mathbb{A} =W/{\sim}_\mathbb{A}$. We write $F_{\mathbb{A}}$ and $R_{\mathbb{A}}$ instead of $F/{\sim}_\mathbb{A}$ and $R/{\sim}_\mathbb{A}$.
 
 \begin{theorem}

$\mathbf{S4DT_0}$ has the finite model property.

\end{theorem}

\begin{proof}

Let $F = (W, R, R_D)$ be an $\mathbf{S4DT_0}$-cone, and a formula $\phi$ is satisfiable in it. We will show that there is a finite $\mathbf{S4DT_0}$-frame in which $\phi$ is satisfiable. First we construct the minimal filtration of $M = (F, V)$ ($\exists x\in\;W\;(M,\;x\vDash \phi)$) via $\sim_\phi$ and denote the resulting model as $M' = (F', V')$, where $F^{\prime} = (W',R',R'_{D})$.{ }The points of $F'$ we will call $equivalence\,\, classes$ or $classes$. Since all the different points of $W$ see each other by the second relation, then each $R'_D$-irreflexive class consists of a single $R_D$-irreflexive point.

For validity of formulas $T_{\square}, D_{\square}, B_D $ and $4_D$ in the resulting frame [see \cite{20}, page 48]. $T_0$  is preserved by minimal filtration, since if there are $R_D'$-irreflexive classes mutually reachable by the first relation we will get a contradiction with the theorem 4.5. The resulting frame may not be $R'$-transitive (Fig.1). To satisfy $4_{\square}$ axiom we can consider transitive closure of the minimal filtration. But if we do this, $T_{\square}, D_{\square}, B_D $ and $4_D$ axioms will still be valid, but $T_0$ may become false. Indeed, consider Fig.1.

\begin{center}
\begin{tikzpicture}
\draw (0,0) ellipse (0.5cm and 0.9cm);
\node [below] at (0,0.5) {$x$};
\node [] at (0,0.5) {\textbullet };

\draw (2.5,0) ellipse (0.5cm and 0.9cm);
\node [] at (2.55, 0.45) {\textbullet {$y$}};
\node at (0,-1.25) {$[x]$};
\node at (2.5, 1.25) {$[y] = [z]$};
\node at (5, 1.25) {$[u]$};
\draw[thick,->] ((0,0.5) -- (2.36, 0.5) node[anchor=north west] {};
\node at (1.2,0.7) {$R$};

\draw (5, 0) ellipse (0.5cm and 0.9cm);
\node [] at (2.55, -0.5) {\textbullet };
\node [below] at (2.55, -0.5) {z};
\draw[thick,->] (2.5, -0.5) -- (4.9, -0.5) node[anchor=north west] {};
\node at (3.74, -0.3) {$R$};
\node [below] at (2.55, -0.5) {$z$};
\node [below] at (5, -0.5) {$u$};
\node [] at (5, -0.5) {\textbullet};
\end{tikzpicture}

Fig. 1

\end{center}

If we take the transitive closure of $R'$ (denote as $R''$) we may have two mutually  reachable by $R''$ $R'_D$-irreflexive classes.

\begin{center}
\begin{tikzpicture}

\draw (0,-0.5) ellipse (0.5cm and 0.9cm);

\node [] at (0,-0.5) {\textbullet };

\draw (2.5,0) ellipse (0.5cm and 0.9cm);
\node [] at (2.5, 0.45) {\textbullet };
\draw[thick,->] ((0,-0.5) -- (2.36, 0.5) node[anchor=north west] {};
\node at (1.2,0.3) {$R$};

\draw (5, 0) ellipse (0.5cm and 0.9cm);
\node [] at (2.5, -0.5) {\textbullet };
\draw[thick,->] ((2.5,-0.5) -- (3.5, -0.5) node[anchor=north west] {};
\node at (3.3, -1.3) {$R$};
\draw[thick,->] ((4, 0.45) -- (4.9, 0.45) node[anchor=north west] {};
\node at (4.4, 0.7) {$R$};
\node [] at (5, 0.45) {\textbullet};
\node [] at (5, -0.5) {\textbullet};
\node [] at (3.8, -0.5) {.\,.\,.\,.\,.};

\draw (7.5,-0.5) ellipse (0.5cm and 0.9cm);
\draw[thick,->] ((5,-0.5) -- (7.4, -0.5) node[anchor=north west] {};
\node at (6.2, -0.3) {$R$};
\node [] at (7.5, -0.5) {\textbullet};

\draw (2.5,-2) ellipse (0.5cm and 0.9cm);
\node [] at (2.5, -2.5) {\textbullet };

\node at (1.2, -1.1) {$R$};

\draw (5, -2) ellipse (0.5cm and 0.9cm);
\node [] at (2.5, -2.5) {\textbullet };
\node [] at (2.5, -1.5) {\textbullet };
\draw[thick,->] ((3.5,-1.5) -- (2.6, -1.5) node[anchor=north west] {};
\node at (3.2, -0.3) {$R$};
\draw[thick,->] ((4.9, -2.5) -- (4, -2.5) node[anchor=north west] {};
\node at (4.25, -2.3) {$R$};
\node [] at (5, -1.45) {\textbullet};
\node [] at (5, -2.5) {\textbullet};
\node [] at (3.6, -2.5) {.\,.\,.\,.};
\draw[thick,->] ((7.41, -0.58) -- (5.1, -1.4) node[anchor=north west] {};
\node at (5.7, -0.94) {$R$};

\draw[thick,->] ((2.5, -2.5) -- (0.1, -0.51) node[anchor=north west] {};
\draw[very thick, ->] (0,0.4) to [out=90,in=90] (7.4,0.4);
\node at (4, 2.8) {$R''$};
\draw[very thick, ->] (7.4,-1.4) to [out=270,in=270] (0,-1.4);
\node at (4, -3.3) {$R''$};

\end{tikzpicture}

Fig.2 
\end{center}

Note that there is a finite number of equivalence classes in $M'$, hence there is a finite number of $R'$-paths (i.e., finite sequences $x_0\, x_1 \dots x_n$ such that $x_iR'x_{i+1}$ for any $i<n$) from one $R'_{D}$-irreflexive class to another, satisfying the following conditions

\begin{itemize}

\item no classes are repeated,

\item there are no $R'_{D}$-irreflexive classes except the beginning and the end of the path.

\end{itemize}

Let $L_1, L_2,\dots, L_m$ are all such paths. $A_{n_1}, A_{n_2}, \dots, A_{n_{n'}}$ are all the  $R'_D$-reflexive classes that appear in $L_1$. We consider all this classes ascending their numbering. Assume that $A_{n_i}$ is visible from class $A_{n_j}$ and sees class $A_{n_k}$ in $L_1$. We devide points of class $A_{n_i}$ into four parts.

1. Points of class $ A_{n_i} $ that are visible from the class $ A_{n_j} $ and see the class $ A_{n_k} $ at the same time will be denoted by $N$.

2. Points of class $ A_{n_i} $ that are visible from the class $ A_{n_j} $ and don't see the class $A_{n_k}$ will be denoted by $ N_1 $. 

3. Points of class $ A_{n_i} $ that are not visible from the class $ A_{n_j} $ and see the class $ A_{n_k} $ are denoted by $ N_2 $. 

4. The last class is $A_{n_i}\setminus(N_1\cup N_2\cup N)$.

\begin{center}

\begin{tikzpicture}

\draw (2.5,-2) ellipse (1cm and 1.8cm);
\draw[thick,->] ((2.5, -2.5) -- (7, 0.2) node[anchor=north west] {};
\node at (2.5, -4.1) {$A_{n_j}$};
\node [] at (2.5, -2.5) {\textbullet};
\node [] at (7.1, 0.21) {\textbullet};
\draw[thick,->] ((1.7, -1.5) -- (6.5, 1.5) node[anchor=north west] {};
\draw[thick,->] ((6.7, 1.45) -- (10.5, -0.5) node[anchor=north west] {};
\node [] at (10.6, -0.515) {\textbullet};
\node [] at (1.7, -1.5) {\textbullet};
\node [] at (6.6, 1.45) {\textbullet};
\draw[thick,->] ((2.6, -1.7) -- (5.8, 0.2) node[anchor=north west] {};
\node [] at (2.6, -1.7) {\textbullet};
\node [] at (5.85, 0.25) {\textbullet};
\draw[thick] ((5.5, 0.5) -- (7.45, 0.5) node[anchor=north west] {};
\node at (6.5, 1) {$N$};
\node at (6.1, 0) {$N_1$};
\draw[thick] ((5.5, -0.3) -- (7.45, -0.3) node[anchor=north west] {};
\draw[thick] ((5.65, -1) -- (7.38, -1) node[anchor=north west] {};
\node at (6.05, -0.7) {$N_2$};
\draw (6.5, 0) ellipse (1cm and 1.8cm);
\node at (6.5, -2.1) {$A_{n_i}$};

\draw[thick, ->] (6.35,-0.5)-- (10, -1) node[anchor=north west] {};
\node at (6.35,-0.5) {\textbullet};\node at (10.05,-1.05) {\textbullet};

\draw[thick, ->] (6.35,-0.9)-- (10, -1.4) node[anchor=north west] {};
\node at (6.35,-0.9) {\textbullet};\node at (10.1,-1.45) {\textbullet};

\draw (10.5,-2) ellipse (1cm and 1.8cm);
\node at (10.5, -4.1) {$A_{n_k}$};

\end{tikzpicture}

\end{center}

\begin{center}

Fig.3

\end{center}

We do this division for all the $n'$ classes. Then we do the same procedure for $L_2, L_3, \dots, L_m$. In result we get divisions for all $R'_D$-reflexive classes $A_1, A_2,\dots, A_{l}$, that appear in the above paths. Then we take the intersection of all divisions of $A_1$ and replace $A_1$ to its subclasses generated by intersection. Then we do the same replacement for all $A_2,\dots, A_l$ and get the frame $F'' = (W'',\, R'',\, R_{D}'' )$.

Then we take the transitive closure of $R''$ (denote $G = (W'',\, S,\, R_{D}'' )$, where $S$ is the transitive closure of $R''$).

We call $A$ $parent\,\, class$ of $a$ if $a\subseteq A$, where $A \in W'$ and $a \in W''$. 

\begin{claim}

$G = (W'',\, S,\, R_{D}'' )$ is an $S4DT_0$ frame. 

\end{claim}

\begin{proof}
We check only $AT_0$ axiom. Assume opposite, i.e. there are mutually reachable by the first relation and irreflexive by the second relation two classes $a$ and $b$ in $G$. Then, there must be a path $a = x_1\, x_2\dots x_n = b$ (or $b = x_1\, x_2\dots x_n = a$), where $x_iR'' x_{i+1}$, for $i = 1,\dots, n-1$, (by the lemma 3.4) satisfying the following conditions:

\begin{itemize}

\item $n > 2$

\item no classes are repeated,

\item there are no classes that have a point that sees a point from right neighbor class and is visible from a point of left neighbor class as in Fig.4.

\end{itemize}

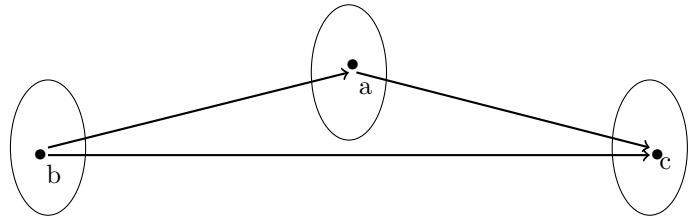
\begin{center}
\begin{tikzpicture}

\draw (-6,-4) ellipse (0.5cm and 0.9cm);

\draw (-2,-3) ellipse (0.5cm and 0.9cm);

\draw (2,-4) ellipse (0.5cm and 0.9cm);

\draw[thick,->] (-6,-4) -- (-2,-3) node[anchor=north west] {a};

\draw[thick,->] (-1.9,-3) -- (2,-4) node[anchor=north west] {c};

\draw[thick,->] (-6,-4.1) -- (2,-4.1) node[anchor=north west] {};

\draw[thick,.] (-6.15 ,-4.1) -- (-6.15,-4.1) node[anchor=north west] {b};

\node [] at (-6.1,-4.1) {\textbullet};

\node [] at (-1.95,-2.9) {\textbullet};

\node [] at  (2.1,-4.1) {\textbullet};

\end{tikzpicture}

Fig. 4

\end{center}

Consider parent classes of $a = x_1\, x_2\dots x_n = b$. Since we didn't devide classes irreflexive by the second relation, the parent classes of $a = x_1, x_2\dots x_n = b$ can be represented as $a = x_1, X_2\dots X_{n-1}\, x_n = b$, where $X_{i}$ is the parent class of $x_i$. 

Now we can delete all circles from (there can be many ways to do this, we do it in some way) $a = x_1\, X_2\dots\,X_{n-1}\, x_n = b$ and get a simple (a path without repeated classes) path from $a$ to $b$. As we don't have any irreflexive by the second relation class between $a$ and $b$, the simple path will have the form $a = x_1, X_{n_1} = X_{2}\dots, X_{n_{n'}} = X_{n-1},\, x_n = b$ (denote this path as $L$), where $1, n_1, n_2, n_{n'}, n$ is subsequence of $1, 2,\dots, n$. 

We have already considered $L$ and devided all classes of this path into 4 parts.

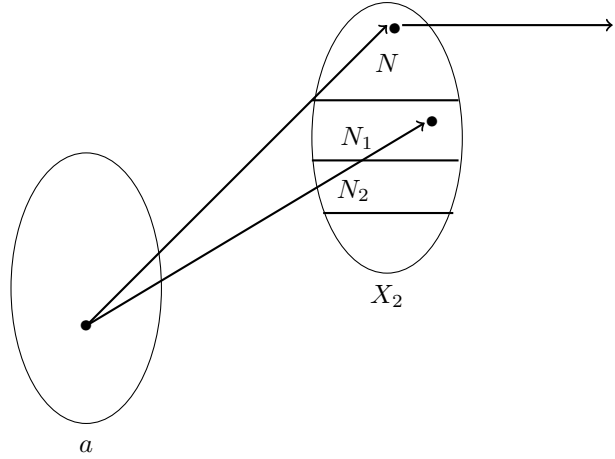
\begin{center}

\begin{tikzpicture}

\draw (2.5,-2) ellipse (1cm and 1.8cm);
\draw[thick,->] ((2.5, -2.5) -- (7, 0.2) node[anchor=north west] {};
\node at (2.5, -4.1) {$a$};
\node [] at (2.5, -2.5) {\textbullet};
\node [] at (7.1, 0.21) {\textbullet};
\draw[thick,->] ((2.5, -2.5) -- (6.5, 1.5) node[anchor=north west] {};
\node [] at (6.6, 1.45) {\textbullet};
\draw[thick,->] ((6.7, 1.5) -- (9.5, 1.5) node[anchor=north west] {};
\draw[thick] ((5.5, 0.5) -- (7.45, 0.5) node[anchor=north west] {};
\node at (6.5, 1) {$N$};
\node at (6.1, 0) {$N_1$};
\draw[thick] ((5.5, -0.3) -- (7.45, -0.3) node[anchor=north west] {};
\draw[thick] ((5.65, -1) -- (7.38, -1) node[anchor=north west] {};
\node at (6.05, -0.7) {$N_2$};
\draw (6.5, 0) ellipse (1cm and 1.8cm);
\node at (6.5, -2.1) {$X_2$};

\end{tikzpicture}

\end{center}

It's obvious that $x_2\subseteq N $ or $x_2\subseteq N_1$ and all the points of $N$ and $N_1$ are visible from $a$. So, $x_2$ have a point that sees a point from right neighbor class and is visible from $a$. This contradiction shows that, $aSb \Rightarrow aRb$. Since $F$ is an $S4DT_0$-cone we finish the proof.

\end{proof}

Then using the Lemma 1.1 we get the end of the proof.

\end{proof}

\section{Completeness of $S4DT_{0}$ with respect to $T_0$-spaces}

\begin{definition}Let $F = (W, R)$ be an $\mathbf{S4}$-frame, then the set $C(x) = R(x) \cap R^{-1}(x)$ for some $x \in W $ is called a $cluster$.
\end{definition}

\begin{definition}

Let $A_1, A_2, \dots, A_n$ be sets of sets. We define $\biguplus_{i} A_{i} = \{a_1 \cup \dots \cup  a_n |\, a_i\in A_i  \}$.

\end{definition}

\begin{theorem}The logic $\mathbf{S4DT_0}$ is complete with respect to topological $T_0$-spaces.
\end{theorem}

\begin{proof}By the lemma 3.4$\;$ each $R$-cluster in an $\mathbf{S4DT_0}$-frame contains no more than one selected point. We know that the logic $\mathbf{S4DT_0}$ has  finite model property, in other words, there is a class $Q$ of finite $\mathbf{S4DT_0}$ cones whose logic is $\mathbf{S4DT_0}$. For each $\mathbf{S4DT_0}$ cone $F \in Q $, we construct a $T_0$-space and a p-morphism from the space to $ Top_{D}(F)$. Consider the following 3 cases:

I. Assume that the cone is a cluster without $R_D$-irreflexive points. As a domain of the space $\mathbb{X}$, we take a countable set of points $X = \{x_1, ..., x_n, ... \}$. We define a topology on $X$ as $T = \{U_n \; | \; n \in \mathbb{N} \} \cup \{\emptyset\}$, where $U_n = \{x_m \; | \; m \geq n  \}$. Let us verify that $\mathbb{X}$ is indeed a topological space:

1. $X, \emptyset \in T$, because $U_1 = X$;

2. $ \bigcup_{k \in I} U_k = U_l $, where $l = min \; I$, $ \forall I \subseteq \mathbb {N} $;

3. $\bigcap_{k \in I} U_k = U_l $, where $l = max \, I$, $I$ is a finite subset of $\mathbb{N} $.

Let $ W = \{w_1, w_2, ..., w_m \} $. We define the map as

\begin{center}

$f(x_{mk + i}) = w_i $,

\end {center}
where $ m $ is the cardinality of $ W $, $i\in \{1, \dots, m \}$ and $ k $ ranges over all natural numbers.

 Function $f$ defined above is a p-morphism, since $f$ is  surjective by the construction, the image of any open set in $ T $ is either the empty set or $ W $, the preimages of the only open sets $ \emptyset $ and $ W $ are either the empty set or $ X $. The set of selected points of the cone is  empty.

II. Let the cone be a cluster with one $R_D$-irreflexive point. As a domain of the space $\mathbb{X}$, we take $X = \{x_1, x_2, ..., x_n, ...\}\bigcup \{$ $\infty\}$. We define a topology on $X$ as $T = \{U'_n \; | \; n \in \mathbb{N} \} \bigcup \{\emptyset$\}, where $U'_n = \{x_m\;|\;m\geq n\}\bigcup\{\infty\}$. Let us verify that $\mathbb{X}$ is indeed a topological space:

1. $X, \emptyset \in T$, because $U'_1 = X$;

2. $ \bigcup_{k \in I} U'_k = U'_l $, where $l = min \; I$, $ \forall I \subseteq \mathbb {N} $;

3. $\bigcap_{k \in I} U'_k = U'_l $, where $l = max \, I$, $I$ is a finite subset of $\mathbb{N} $.

Let $ W = \{w_0, w_1, ..., w_m \}$, where $ w_0 $ is $R_D$-irreflexive point. We define the map as

\begin {center}

$f (x_{mk + i}) = w_i,\, i = 1, 2,\dots, m $,

\end {center}

\begin {center}

$f (\infty) = w_0$.

\end {center}
where $m$ is the cardinality of $W\backslash\{m_{0}\}$ and $k$ ranges over all natural numbers.

Function $f$ defined above is a p-morphism, since $f$ is  surjective by the construction, the image of any open set in $ T $ is either an empty set or $ W $, the preimages of the only open sets $ \emptyset $ and $ W $ are either an empty set or $ X $ and  $f^{-1}(w_0) = \{ \infty\}$ is irreflexive singleton.

III. Let us consider a general case. Let frame $ F = (W, R, R_ {D}) $ be a cone. The preorder $R$ induces an equivalence relation on $F:$

\begin{center}

$x\sim y \Leftrightarrow xRy \wedge yRx$.

\end{center}

It identifies points belonging to the same cluster. Let $F' = (W', R') = (W, R)/\sim$ is called the $skeleton$ of $F$. Let $Top(F') = (Y, T_Y)$.

Now we can construct the required topological space and define required p-morphism. Let us construct disjoint topological spaces for each cluster according to cases I and II i.e., $ \mathbb{X}_i = (X_i, T_i)$ and the p-morphisms $f_i$ in the case of clusters. The domain of the space is $X =\bigcup_{i\in I} X_i$, where $I$ is the set of all clusters, $X_i$ is the domain of the corresponding space. We define the topology as $T_X =\{\emptyset\}\cup \bigcup_{U\in T_{Y}} O_U $, where $O_U = \biguplus_{a\in U} (T_{a}\backslash\{\emptyset \})$. 
 
Each element $V$ of $T_X$ can be represented as $V = \bigcup_{a\in U} U_{a}$,  where $U$ is an open set in $Top(F')$ that $corresponds$ to $V$ and $U_{a}$ is a nonempty open set from the topological space $X_a$, corresponding to the cluster $a$.

\begin{claim}

$T_X $ is a topology.

\end{claim}

\begin{proof}

That $\emptyset, \, X \in  T_X$ is obvious.

Suppose $I$ is the set of all clusters and $\{U_j : j\in J\}\subset T_X$. Each open set $U_j$, where $j\in J$ corresponds to $V_j$, which is an open set in $T_Y$. Then, $\bigcup_{j\in J} V_j = V$ is open. Then $\bigcup_{j\in J} U_j = \bigcup_{a_k \in V}\bigcup_{U_j\in J} (X_{a_k}\cap U_j)$, where $X_{a_k}$ is the domain of the space that corresponds to $a_k\in V$. For each $a_k \in V$ the set $\bigcup_{U_j\in J} (X_{a_k}\cap U_j)$ is a nonempty open set in $T_{a_{k}}$. In the end, from the definition of open sets in $T_X$, we conclude that $\bigcup_{j\in J} U_j$ is open in $T_X$ .

Then, assume $U'$ and $U''$ are open in $T_X$ and each $U'$ and $U''$ corresponds to $V'$ and $V''$, which are open in $T_Y$. Then, $U'\cap U'' = \bigcup_{a_j\in V }((U'\cap X_{a_j}) \cap (U''\cap X_{a_j}))$ is open, where $V = V'\cap V''$ and $X_{a_j}$ is the domain of the space that corresponds to $a_j\in V$.

\end{proof}

\begin{claim}

$\mathbb{X}$ is a $T_0$ space.

\end{claim}

\begin{proof}

Consider two points $x \ne y$. If their image under the map $f$ falls into the same cluster $a$, then there is a space $(X_a, T_a)$ corresponding to this cluster which contains these points.  Since this space is a $T_0$ space, then there exists a $U_a\in T_a$ that contains only one of these points. Then we take $\bigcup_{y\neq a\, \&\, y \in Y} X_y \,\,\, \cup U_{a}$.

In the case when the points $x$ and $y$ lie on different spaces $\mathbb{X}_a = (X_a, T_a)$ and $ \mathbb{X}_b = (X_b, T_b)$ correspondingly, we consider points $c_a, c_b\in F'$ which correspond to $\mathbb{X}_a$ and $\mathbb{X}_b$. Since $R'$ is partial order, then $\neg (c_aR'c_b)$ or $\neg (c_bR'c_a)$. In the first case $\cup_{c\in R'(c_a)} X_c$ is an open set, that contains $x$ but doesn't contain $y$. The second case is treated similarly.

\end{proof}

We define $f$ as the union of the maps $ f_i $.

\begin{claim}

$f : \mathbb{X}\twoheadrightarrow Top_D(F)$.

\end{claim}

\begin{proof}

Surjectivity of $f$ follows from surjectivity of each $f_i$. 

Suppose $V$ is open in $Top_D(F)$. $V$ corresponds to an open set $U$ in $Top(F')$ then $f^{-1}(V) = \cup_{a\in U} X_a$ that is open in $\mathbb{X}$.

Assume $U$ is open in $\mathbb{X}$ and corresponds to an open set $V$ in $Top(F')$ (i.e. $U = \cup_{a\in V}\, U'_a$, where $U'_{a}$ is open in $\mathbb{X}$). Then, $f(U) = \bigcup_{a\in V} f(U'_{a}) = \bigcup_{a\in V} f_{a}(U'_{a})$, where $f_{a}(U'_{a})$ is cluster. Then $R'(V)\subseteq V \Rightarrow R(U)\subseteq U$, i.e., $U$ is open in $Top_D(F)$.

\end{proof}

So, we have constructed a topological space for each cone in $Q$, and then a corresponding p-morphism. Further, by the lemma 4.1 and by the theorems 2.1, 2.2 and 5.1 we obtain the assertion of the theorem.

\end{proof}

\section{References}
\label{S.7}


\end{document}